\documentclass{article}
\usepackage{amsmath,amssymb}
\usepackage{enumerate}

\newtheorem{theorem}{Theorem}[section]
\newtheorem{lemma}{Lemma}[theorem]
\newtheorem{proposition}[theorem]{Proposition}
\newtheorem{corollary}[theorem]{Corollary}
\newtheorem{definition}[theorem]{Definition}

\newtheorem{question}[theorem]{Question}
\newtheorem{remark}[theorem]{Remark}

\newcommand{\Z}{\mathbb{Z}}

\newcommand{\id}{\operatorname{id}}
\newcommand{\Id}{\operatorname{Id}}

\newcommand{\aut}{\operatorname{Aut}}

\newcommand{\hol}{\operatorname{Hol}}
\newcommand{\ord}[1]{\left\vert #1\right\vert}
\newcommand{\lie}[1]{\left[ #1\right]}

\newcommand{\mlie}[2]{#1#2-#2#1}

\newenvironment{proof}{\par\noindent{\bf Proof.}}{$\qed$\par\bigskip}
\newcommand{\qed}{\enspace\vrule  height6pt  width4pt  depth2pt}

\title{Counterexample to a conjecture about braces}
\author{David Bachiller\footnote{Research partially supported by a grant of MICIIN (Spain)
MTM2011-28992-C02-01, and MTM2014-53644-P.}}
\date{}

\begin{document}
\maketitle

\abstract{We find an example of a finite solvable group (in fact, a finite $p$-group) without 
any left brace structure (equiv. which is not an IYB group). Our argument is an improvement of an argument of Rump, 
using previous work in other areas of Burde, and of Featherstonhaugh, Caranti and Childs, 
which we relate to brace theory.}

{\bf Keywords:} braces, IYB group, bijective $1$-cocycle, radical rings, Hopf-Galois extensions, nilpotent group, Lie algebras, 
LSA structures

{\bf MSC:} Primary 16N20, 20D15, 81R50; Secondary 16T25, 12F10, 17B30, 17B60.

\section{Introduction}

In the last thirty years, the study of the Yang-Baxter equation has increased tremendously, motivated
both by its applications in theoretical physics, and by
its connections to many topics in mathematics. The construction of all the solutions of this equation is a widely open 
problem, so one tries to study particular classes of solutions. One of these classes, the non-degenerate involutive set-theoretical 
solutions, has received a lot of attention recently (see the introduction of \cite{CJO}, and the references there). 

The study of this class of solutions was initiated in \cite{ESS} and \cite{GVdB}. One of the techniques used in 
\cite{ESS} was to associate a group to any non-degenerate involutive set-theoretical 
solution, called the permutation group of the solution. This group carries a lot of information
about the solutions, and it has been very important to study them.
In fact, in \cite{CJR} (where the finite permutation groups are called IYB groups), it was proposed to study this class of groups as
a first step to construct and classify non-degenerate involutive set-theoretical 
solutions of the Yang-Baxter equation. 

It would be very useful to characterize which groups are IYB groups, independently of showing a particular solution.  
It is known that any IYB group is a solvable group \cite[Theorem 2.15]{ESS}, 
and this raises the question whether the converse is true:

\begin{question}\label{conjecture}
Let $G$ be a finite solvable group. Is $G$ an IYB group?
\end{question}

In order to study the class of non-degenerate involutive set-theoretic solutions, 
Rump in \cite{Rump1} introduced a new algebraic structure called brace. 
A left brace is a set $B$ with two operations, a sum $+$ and a multiplication $\cdot$, such that $(B,+)$ is an abelian group, 
$(B,\cdot)$ is a group, and for any $a,b,c\in B$, $a\cdot (b+c)+a=a\cdot b+a\cdot c$. A right brace is defined analogously, changing 
the last property by $(b+c)\cdot a+a=b\cdot a+c\cdot a$. A left and right brace is called a two-sided brace. 

It turned out (see \cite[Theorem 2.1]{CJR}) that 
a finite group is an IYB group if and only if it is the multiplicative of a left brace, and that this 
additional algebraic structure over IYB groups is helpful to attack some problems. Hence we can reformulate 
Question \ref{conjecture} in terms of left braces as:

\paragraph{Question \ref{conjecture} (equivalent reformulation)} \textit{Let $G$ be a 
finite solvable group. Is $G$ the multiplicative group of a left brace?}
\bigskip

Some results supporting a positive answer to the question can be found in \cite{NG,CJR,CJO2,Florian}. 
In \cite{NG,CJR}, it is proved that some classes of finite solvable groups are the multiplicative group of a left brace; namely, 
abelian groups, nilpotent groups of class $2$, Hall subgroups of the multiplicative group of a left brace, abelian-by-cyclic groups,
 and solvable A-groups (i.e. solvable groups such that all their Sylow $p$-subgroups are abelian).  In \cite{CJR}, 
it is also proved that any finite solvable group is isomorphic to a subgroup of the multiplicative group of a 
left brace.
In \cite{CJO2}, there are also positive results for some nilpotent groups of class $3$, and for some metabelian groups, 
but the answer in general for nilpotent groups of class greater or equal than $3$ and for metabelian groups is not 
known. In \cite{Florian}, there is explained a computer procedure that shows that any solvable group of order $\leq 200$ and 
any $p$-group of order $<1024$ is the multiplicative group of a left brace.

In spite of all these positive results, in this paper we answer Question \ref{conjecture} in the negative, by 
providing the first known example of a $p$-group which is not the multiplicative group of any left brace. To find this group, we use the 
following argument, suggested in \cite[Section 12]{Rump}:
for any finite left brace $B$ with additive group isomorphic to $A$, there exists an injective morphism from $(B,\cdot)$ to 
$A\rtimes\aut(A)$ (see Proposition \ref{monomorph}). So if we are able to find a finite solvable group $G$ with no injective morphism 
to $A\rtimes \aut(A)$ for any abelian group $A$ such that $\ord{G}=\ord{A}$, then $G$ is an example of negative answer to Question 
\ref{conjecture}. 

In \cite{Rump}, it is suggested as a possible counterexample an $11$-group $G$ of order $11^{10}$ coming from the theory of Lie algebras. But 
in that paper one can only find a vague argument to see that there is no injective morphism from 
$G$ to $A\rtimes\aut(A)$ for $A=(\Z/(p))^{10}$, and 
there is no argument for the other abelian groups with $\ord{A}=\ord{G}$. 
Hence in \cite[Section 12]{Rump} there is only a possible idea to find a counterexample, but the argument is not complete for two reasons:
\begin{enumerate}[(a)]
\item There is no argument for the other additive groups not isomorphic to $(\Z/(p))^{10}$, and
\item The argument for $A=(\Z/(p))^{10}$ depends on the fact that some computer calculations of \cite{Burde} over $\mathbb{C}$ remain true 
over $\mathbb{F}_p$, for $p=11$. This is far from clear, and no computer programs are provided in \cite{Rump}. 
\end{enumerate}
In fact, the computer programs of \cite{Burde} do not work for $p=11$ (see Remark \ref{errorRump}), 
so the vague argument in \cite{Rump} for $A=(\Z/(p))^{10}$ is not even 
correct. 

In this article,
we complete points (a) and (b) for a suitable $p$ to find our counterexample. For (a), we prove 
some general results about the possible additive group of a left brace of order a power of $p$. In fact, 
for $p$ big enough, we prove that there is only one possible additive group, and that allows us to 
restrict the argument to the case of additive group isomorphic to $(\Z/(p))^{10}$. For (b), 
we manage to prove (theoretically) that the computer calculations of \cite{Burde} over $\mathbb{C}$ 
remain true over $\mathbb{F}_p$ if $p$ is big enough. So we settle (b) for some $p$ without using a computer.
However, observe that the argument still depends in the calculations of \cite{Burde} over $\mathbb{C}$.

The organization of the paper is as follows: first, Section 2 contains the results about the additive group 
of a left brace that are necessary to complete (a) for our case. These results are generalizations of results in the 
theory of Hopf-Galois extensions, and at the end of this section we explain the connexion between this type of Hopf
algebras and brace theory. 
Then, in Section 3 we recall the Lazard's correspondence, which is useful to translate results about nilpotent Lie algebras 
to results about $p$-groups. We need it because, in Section 4, we present our counterexample as a nilpotent Lie algebra 
over $\mathbb{F}_p$, and then we apply the Lazard's correspondence to obtain it in group form. We follow this procedure 
because the calculations of \cite{Burde} for characteristic $0$ come from a Lie algebra over $\mathbb{C}$. In Section 4, 
we prove that these computations are also correct in the characteristic $p$ case, for $p$ big enough, in a theoretical way. 
We hope that arguments of this kind, relating results over $\mathbb{C}$ and over $\mathbb{F}_p$, will be useful in the future
in brace theory.

\section{Restrictions over the additive group of a finite left brace}

\begin{definition}
A left brace is a set $B$ with two binary operations, a sum $+$ and a multiplication $\cdot$, such that
$(B,+)$ is an abelian group, $(B,\cdot)$ is a group, and any $a,b,c\in B$ satisfies
$$
a\cdot (b+c)+a=a\cdot b+a\cdot c.
$$
\end{definition}

In a left brace $B$, for each $g\in B$ define a map $\lambda_g: B\to B$ by $\lambda_g(h):=g\cdot h-g$. 
Then, $\lambda_g$ is an automorphism of $(B,+)$ for every $g$, and the map $\lambda: (B,\cdot)\to\aut(B,+)$, 
$g\mapsto \lambda_g$, is a morphism of groups (see \cite[Lemma 1]{CJO}).  

\begin{definition}
Let $G$ be a group. The holomorf of $G$, denoted by $\hol(G)$, is defined as
$$
\hol(G):=G\rtimes \aut(G).
$$
\end{definition}

Now we restrict to the case of $\hol(A)$ for some abelian group $A$. 
 The following result is an easy generalization of \cite[Theorem 1]{CR}, which gives an equivalence between left braces 
and regular subgroups of the holomorf. 
Recall that a regular subgroup of $\hol(A)$ is a subgroup $H\leq \hol(A)$ such that for any $w\in A$ there exists a unique 
$(v,M)\in H$ such that $(v,M)(w)=v+M(w)=0$.

Denote by $\pi_1$ and $\pi_2$ the two maps $\pi_1:\hol(A)\to A$, $\pi_1(v,M)=v$, and $\pi_2:\hol(A)\to \aut(A)$, $\pi_2(v,M)=M$. 

\begin{proposition}\label{transl}
Let $A$ be an abelian group.
\begin{enumerate}[(1)]
\item Let $B$ be a left brace with additive group $A$. Then, $\{(a,\lambda_a):a\in A\}$ is a 
regular subgroup of $\hol(A)$.

Conversely, if $H$ is a regular subgroup of $\hol(A)$, we have $\pi_1(H)=A$, and 
the abelian group $A$ with the product
$$a\cdot b:=a+\pi_2((\left.\pi_1\right|_H)^{-1}(a))(b)$$
is a left brace with multiplicative group isomorphic to $H$.

\item This defines a bijective correspondence between left braces with additive group $A$, and regular
subgroups of $\hol(A)$. Moreover, 
isomorphic left braces correspond to conjugate subgroups of $\hol(A)$ by elements of $\aut(A)$.
\end{enumerate}
\end{proposition}
\begin{proof}

To prove (1), observe that $(0,\id)=(0,\lambda_0)$, that $(a,\lambda_a)(b,\lambda_b)=(a+\lambda_a(b),\lambda_a\lambda_b)=
(ab,\lambda_{ab})$, and that $(a,\lambda_a)^{-1}=(-\lambda_a^{-1}(a),\lambda_a^{-1})=(a^{-1},\lambda_{a^{-1}})$, so 
$H=\{(a,\lambda_a):a\in A\}$ is a subgroup of $\hol(A)$. To check that it is regular, given 
$b\in A$, take $a=b^{-1}$. Then $(a,\lambda_a)(b)=a+\lambda_a(b)=a\cdot b=b^{-1}\cdot b=0$, and $(a,\lambda_a)$ with this property 
is unique because the inverse element of $b$ in $(B,\cdot)$ is unique. 

Conversely, if $H$ is a regular subgroups of $\hol(A)$, the regularity implies that, for all $a\in A$, there exists 
a unique $(v,M)\in H$ such that $a=(v,M)(0)=v+M(0)=v$. Hence $\pi_1(H)=A$. Then we can write any element of $H$ has $(a,\phi(a))$
for some map $\phi: A\to\aut(A)$.
Since $(a,\phi(a))(b,\phi(b))=(a+\phi(a)(b),\phi(a)\phi(b))$, the map $\phi$ satisfies $\phi(a)\phi(b)=\phi(a+\phi(a)(b))$. 
We define a product over $A$ by
$$
a\cdot b:=a+\pi_2((\left.\pi_1\right|_H)^{-1}(a))(b)=a+\phi(a)(b).
$$
To check that it defines a structure of left brace is a straightforward exercise. 

In (2), the bijective correspondence is clear by (1). Then, if $B_1$ and $B_2$ are two left braces with additive group 
equal to $A$ such that $\varphi: B_1\to B_2$ is an isomorphism of left braces, note that $\varphi$ is in particular an 
automorphism of $A$. Then, in $\hol(A)$, $B_1$ and $B_2$ are conjugate by $(0,\varphi)$ because
\begin{eqnarray*}
(0,\varphi)(g,\lambda^{(1)}_g)(0,\varphi)^{-1}&=&(\varphi(g),\varphi\lambda^{(1)}_g)(0,\varphi^{-1})=
(\varphi(g),\varphi\lambda_g^{(1)}\varphi^{-1})\\
&=&(\varphi(g),\lambda_{\varphi(g)}^{(2)}).
\end{eqnarray*}
The converse is analogous.
\end{proof}

The interesting part for our purposes is the following corollary of the last proposition.

\begin{proposition}\label{monomorph}
Let $B$ be a left brace. Then, the map
$$
\begin{array}{cccc}
\gamma:& (B,\cdot)& \to & \hol(B,+)\\
	&  g&		\mapsto &	(g,~\lambda_g),
\end{array}
$$
is a monomorphism of groups. 
\end{proposition}
\begin{proof}
It is clear that the map is injective (since $(g,\lambda_g)=(h,\lambda_h)$ implies in particular $g=h$), 
so the only thing to check is $\gamma(gh)=\gamma(g)\gamma(h)$, for all $g,h\in B$. 
But this is a combination of the definition of the product in a semidirect product and properties of 
the lamba map in left braces gives
$$
\gamma(g)\gamma(h)=(g,\lambda_g)(h,\lambda_h)=(g+\lambda_g(h),\lambda_g\lambda_h)=(gh,\lambda_{gh})=\gamma(gh).
$$
\end{proof}

Let $B$ be a finite left brace, and let $x\in B$. We denote by $o_{\cdot}(x)$ the multiplicative order of $x$,
 and we denote by $o_{+}(x)$ the additive order of $x$. 
The following result gives some restrictions over the additive group that a left brace can have.

\begin{theorem}\label{feather}
Let $p$ be a prime number, and $m$ a positive integer. Let $B$ be a finite left brace with 
$$
(B,+)\cong \Z/(p^{\alpha_1})\times\cdots\times\Z/(p^{\alpha_m}),
$$
for certain $\alpha_i\in\Z$ such that $1\leq\alpha_1\leq \alpha_2\leq\dots\leq\alpha_m$. 
Assume that $m+2\leq p$. Then, $o_{\cdot}(x)=o_{+}(x)$ for any $x\in B$. 
Moreover, if $(B,\cdot)$ is abelian, then $(B,\cdot)\cong (B,+)$.

In particular, if $\ord{B}=p^n$ and $n+2\leq p$, then $o_{\cdot}(x)=o_{+}(x)$ for any $x\in B$.
\end{theorem}

To prove this theorem, first we need two lemmas.

\begin{lemma}\label{lemmaF}
Let $p$ be a prime number, and let $B$ be a finite left brace with 
$(B,+)\cong (\Z/(p))^m$. Assume $m+1\leq p$. Then, $o_{\cdot}(x)=o_{+}(x)=p$ for any $x\in B$.
\end{lemma}
\begin{proof}
We know that we can define a monomorphism $\gamma: (B,\cdot)\hookrightarrow \hol(B,+)= (\Z/(p))^m\rtimes \aut((\Z/(p))^m)\cong (\Z/(p))^m\rtimes 
GL_m(\mathbb{F}_p)$. Observe that 
$$
U_m(\mathbb{F}_p)=\{\Id+N\in GL_m(\mathbb{F}_p) : N\text{ is a strictly upper triangular matrix}\}
$$
is a Sylow $p$-subgroup of $GL_m(\mathbb{F}_p)$, so
$(\Z/(p))^m\rtimes U_m(\mathbb{F}_p)$ is a Sylow $p$-subgroup of $(\Z/(p))^m\rtimes GL_m(\mathbb{F}_p)$.
 Since $(B,\cdot)$ is a $p$-group, after a suitable conjugation, we can think that $\gamma(B)$ is a subgroup 
 of $(\Z/(p))^m\rtimes U_m(\mathbb{F}_p)$. But note also that $(\Z/(p))^m\rtimes U_m(\mathbb{F}_p)\cong U_{m+1}(\mathbb{F}_p)$ 
 by the isomorphism
 $$
 (v,M)\longmapsto \begin{pmatrix} M& v\\ 0& 1\end{pmatrix}.
 $$
In conclusion, we can assume without lost of generality that $(B,\cdot)\leq U_{m+1}(\mathbb{F}_p)$.

Now we will see that, when $p>m$, any element of $U_{m+1}(\mathbb{F}_p)$ has multiplicative order $p$.
Take $\Id+N\in U_{m+1}(\mathbb{F}_p)$, with $N$ strictly upper triangular. Then,
$$
(\Id+N)^p=\Id+N^p=\Id+N^{m+1}\cdot N^{p-m-1}=\Id,
$$
using in the first equality that we are in characteristic $p$, and using in the third and in the last one 
that $p>m$, and that $N^{m+1}=0$ because it is a strictly upper triangular matrix of size $m+1$. 

In particular, any element of $(B,\cdot)\leq U_{m+1}(\mathbb{F}_p)$ has order $p$.
\end{proof}

\begin{lemma}\label{lemma2}
Let $A=\Z/(p^{\alpha_1})\times\cdots\times\Z/(p^{\alpha_m})$, for certain $\alpha_i\in\Z$ 
such that $1\leq\alpha_1\leq \alpha_2\leq\dots\leq\alpha_m$. Then, 
for any $M\in \aut(A)$ with order equal to a power of $p$, $(M-\Id)^m=pN$ for some endomorphism $N$ of $A$.  
\end{lemma}
\begin{proof}
Any $M\in\aut(A)$ determines an automorphism $\overline{M}$ of $A/pA\cong (\Z/(p))^m$. 
We can think the elements of $\aut((\Z/(p))^m)$ as matrices in $GL_m(\Z/(p))$. 
Since $\overline{M}$ has a prime power order, it is conjugate to an upper triangular matrix with $1$'s in the diagonal, 
and then $(\overline{M}-\overline{\Id})^m=\overline{0}$, because the matrix is of size $m$. 
Going back to $A$, this implies $(M-\Id)^m=pN$ for some endomorphism $N$ of $A$. 
\end{proof}

\begin{proof}{ \bf(of Theorem \ref{feather})}
Denote $\Omega_i(B,+)=\{x\in B : p^ix=0\}$. Also, denote $\Omega_i(B,\cdot)=\{x\in B : x^{p^i}=0\}$. It is enough to 
show that $\Omega_{i+1}(B,+)\setminus \Omega_i(B,+)\subseteq \Omega_{i+1}(B,\cdot)\setminus \Omega_i(B,\cdot)$ for any 
$i\geq 0$, because 
$B$ is the disjoint union of $\{0\}$ and the sets in the left-hand side (resp. right-hand side), and then equality follows. 

We are going to prove these inclusions by induction over $i$. For $i=0$, observe that $T=\Omega_1(B,+)$ is a sub-brace of $B$ with
additive group isomorphic to an elementary abelian group of $p$-rank equal to $m$, because the $p$-rank of $(B,+)$ is $m$. Then, by 
Lemma \ref{lemmaF}, we know that the exponent of $(T,\cdot)$ is equal to $p$, so $T=\Omega_1(B,+)\subseteq \Omega_1(B,\cdot)$.

Now assume $i\geq 1$, and assume it is true that
$$
\Omega_{i}(B,+)\setminus \Omega_{i-1}(B,+)\subseteq \Omega_{i}(B,\cdot)\setminus \Omega_{i-1}(B,\cdot).
$$
We shall prove that $\Omega_{i+1}(B,+)\setminus \Omega_i(B,+)\subseteq \Omega_{i+1}(B,\cdot)\setminus \Omega_i(B,\cdot)$.
Let $a\in \Omega_{i+1}(B,+)\setminus \Omega_i(B,+)$. We show first that $a^p\in\Omega_i(B,+)$. Note that
\begin{eqnarray*}
a^p&=&(\Id+\lambda_a+\lambda_a^2+\cdots+\lambda_a^{p-1})(a)\\
&=& pa+\sum_{i=2}^{p-1} \binom{p}{i} (\lambda_a-\Id)^{i-1}(a)+(\lambda_a-\Id)^{p-1}(a)
\end{eqnarray*}
Since $pa\in\Omega_{i}(B,+)$, $pa+\sum_{i=2}^{p-1} \binom{p}{i} (\lambda_a-\Id)^{i-1}(a)\in \Omega_i(B,+)$, and 
moreover $(\lambda_a-\Id)^{p-1}(a)\in\Omega_i(B,+)$ because, by Lemma \ref{lemma2}, $p-1\geq m+1\geq m$ implies that 
$(\lambda_a-\Id)^{p-1}=pN$ for some endomorphism $N$ of $(B,+)$. Hence $p^i(\lambda_a-\Id)^{p-1}(a)=p^ipN(a)=N(p^{i+1}a)=0$. 
Therefore $a^p\in\Omega_i(B,+)$.

Now we prove that $a^p\not\in \Omega_{i-1}(B,+)$. 
Define the abelian group $T=\Omega_{i+1}(B,+)/\Omega_{i-1}(B,+)$.
Denote by 
$$S:=\left\{\left.\sum_{i=1}^r \overline{(\lambda_a-\Id)^{k_i}(a)}\in T ~\right\vert~r\geq 0, ~k_i\geq 0\right\}.$$
Note that $S$ is a subgroups of $T$.  
More generally, denote by 
$$S^k:=\left\{\left.\sum_{i=1}^r \overline{(\lambda_a-\Id)^{k_i}(a)}\in T ~\right\vert~r\geq 0, ~k_i\geq k-1\right\},$$  
and note that $S^{k+1}$ is a subgroup of $S^k$ for any $k\geq 1$, and that $(\lambda_a-\Id)(s)\in S^{k+1}$ for any $s\in S^{k}$. 
Besides, $S^\alpha=0$ for a sufficiently big $\alpha$, since 
$(\lambda_a-\Id)^m=pN$ by Lemma \ref{lemma2} implies that $(\lambda_a-\Id)^{m\alpha_m}=p^{\alpha_m}N^{\alpha_m}=0$. We know that 
$pa\not\in \Omega_{i-1}(B,+)$, which means
$\overline{pa}\neq \overline{0}$ in $T$, so 
$\overline{pa}\in S$, but $\overline{pa}\not\in S^\beta$ for some power of $S$. Assume that 
$\overline{pa}\in S^k$ but $\overline{pa}\not\in S^{k+1}$. Then, 
$S/S^k$ is an elementary abelian group because $S^{k}$ contains $\overline{pa}$, and hence 
$p\overline{(\lambda_a-\Id)^{l}(a)}=\overline{(\lambda_a-\Id)^{l}(pa)}\in S^{k}$ for any $l\geq 0$.
 
The elementary abelian group $S/S^k$ has dimension (as $\mathbb{F}_p$-vector space) equal to $k-1$, since it has a basis
consisting of the classes of the elements $a,~(\lambda_a-\Id)(a),\dots,(\lambda_a-\Id)^{k-2}(a)$ in $S/S^k$: they generate
$S/S^k$, and if $\gamma_0 a+\gamma_1 (\lambda_a-\Id)(a)+\cdots +\gamma_{k-2}(\lambda_a-\Id)^{k-2}(a)=0$ in 
$S/S^k$ for some $\gamma_0,\dots,\gamma_{k-2}\in\{0,1,\dots,p-1\}$, then 
$\gamma_0 a+\gamma_1 (\lambda_a-\Id)(a)+\cdots +\gamma_{k-2}(\lambda_a-\Id)^{k-2}(a)=(\lambda_a-\Id)^{k-1}(s)$ in $S$ for 
some $s\in S$. Applying $(\lambda_a-\Id)$ to this equality enough times, we get $\gamma_0(\lambda_a-\id)^{\beta}(a)=0$, with 
$(\lambda_a-\id)^{\beta}(a)\neq 0$; thus $\gamma_0=0$. Repeating this process, we get $\gamma_0=\dots=\gamma_{k-2}=0$, so 
$a,~(\lambda_a-\Id)(a),\dots,(\lambda_a-\Id)^{k-2}(a)$ are linearly independent. 
Now recall that $S/S^k$ is a section of $(B,+)$, so its $p$-rank is less than or equal to $m$, 
the $p$-rank of $(B,+)$. This yields the 
inequalities $k-1\leq m\leq p-2$, which implies $k+1\leq p$. 

Now recall that 
$$
a^p = pa+\sum_{i=2}^{p-1} \binom{p}{i} (\lambda_a-\Id)^{i-1}(a)+(\lambda_a-\Id)^{p-1}(a).
$$
Since $\overline{pa}\in S^k$, we have that $\sum_{i=2}^{p-1} \binom{p}{i} \overline{(\lambda_a-\Id)^{i-1}(a)}\in S^{k+1}$. Moreover 
$$\overline{(\lambda_a-\Id)^{p-1}(a)}\in S^p\subseteq S^{k+1},$$ because $k+1\leq p$. Since $\overline{pa}\not\in S^{k+1}$, 
we have that $\overline{a^p}\not\in S^{k+1}$, which implies that $a^p\not\in \Omega_{i-1}(B,+)$.

For the moment, we know that $a^p\in\Omega_{i}(B,+)\setminus \Omega_{i-1}(B,+)$. By induction hypothesis, 
$a^p\in\Omega_{i}(B,+)\setminus \Omega_{i-1}(B,+)\subseteq \Omega_{i}(B,\cdot)\setminus \Omega_{i-1}(B,\cdot)$, 
and this implies finally that $a\in \Omega_{i+1}(B,\cdot)\setminus \Omega_{i}(B,\cdot)$. 
\end{proof}

Theorem \ref{feather} is inspired in \cite[Proposition 5]{FCC}. In \cite{FCC}, the proof is in terms of radical rings, 
but remember that two-sided
braces and radical rings are equivalent by \cite[Proposition 1]{CJO}. Observe that they prove this theorem for braces with 
abelian multiplicative group, and we have generalized their proof to general left braces. 
Another inspiration for this result is \cite[Theorem 3.2]{Feather}.

\begin{remark}\rm{
We explain now the relation between brace theory and the theory of Hopf-Galois extensions, that 
has been useful to apply results like the ones in \cite{Feather} and \cite{FCC} in our paper. Let $L/K$ be a finite 
field extension. We say that $L/K$ is a Hopf-Galois extension if there exists a Hopf algebra $H$ over $K$
of finite dimension, and $\mu:H\to \operatorname{End}_K(L)$ a Hopf action such that 
$
(1,\mu): L\otimes_K H\to \operatorname{End}_K(L)
$
is an isomorphism of $K$-vector spaces, where $(1,\mu)(l\otimes h)(t)=l\cdot (\mu(h)(t))$. 
For example, when $L/K$ is a Galois extension with Galois group $G$, the Hopf algebra
$H=K[G]$ satisfies these properties. 

It is proved in \cite{By} that, when $L/K$ is a finite Galois extension and $G=Gal(L/K)$, 
if $G'$ is a group such that there exists an injective morphism of groups $\gamma: G\hookrightarrow\hol(G')=G'\rtimes\aut(G')$
satisfying that $\gamma(G)$ is a regular subgroup of $\hol(G')$, then $H=K[G']$ defines a Hopf-Galois extension of 
$L/K$. Moreover, it is proved that any Hopf-Galois extension is of this form. Hence this translates the problem 
of finding Hopf-Galois extensions completely in group-theoretical terms: given a Galois group $G$, first find 
all the regular subgroups of $\hol(G')$ isomorphic to $G$, and second, find all the injective morphisms from $G$ 
to $\hol(G')$ with one of these subgroups as image. Observe that the regularity property implies that $\ord{G}=\ord{G'}$.

Note that, by Proposition \ref{transl}, finding regular subgroups of $\hol(G')$ when $G'$ is abelian is equivalent to 
find left braces with additive group isomorphic to $G'$. 
So the first part of the problem of construction of abelian Hopf-Galois extensions (i.e. with $H=K[G']$ commutative) 
of a Galois extension $L/K$ with Galois group equal to $G$ is equivalent to 
our problem of construction of left braces with multiplicative group isomorphic to $G$. 

We hope that this connection between this two theories would be fruitful in the future.  
For instance, in terms of Hopf-Galois extensions, \cite[Proposition 5]{FCC} is proved for $G$ and $G'$ abelian 
and, using brace theory, we have managed to generalize it for $G'$ abelian and $G$ non-abelian in Theorem \ref{feather}. }
\end{remark}

\section{Lazard's correspondence}

Assume that we have proved the following result.

\begin{theorem}\label{countergroup}
For some primer $p\geq 12$, there exists a group $(G,\cdot)$ of order $p^{10}$ and exponent $p$ without any monomorphism
$(G,\cdot)\hookrightarrow GL_{11}(\mathbb{F}_{p})$.
\end{theorem}

Then, Question \ref{conjecture} is answered in the negative: $G$ is a solvable group, and suppose to arrive to a contradiction 
that we can define over it a structure of left brace. Then, since $p\geq 10+2=12$, using Theorem \ref{feather},
we know that $o_{+}(x)=o_{\cdot}(x)=p$ for any $x\in G\setminus\{e\}$ (because the multiplicative exponent of $G$ is $p$).
Thus $(G,+)\cong (\Z/(p))^{10}$. Using Proposition \ref{monomorph}, we can then define a monomorphism 
$(G,\cdot)\hookrightarrow \hol(G,+)\cong \hol((\Z/(p))^{10})$. But there is also a monomorphism 
$\hol((\Z/(p))^{10})\hookrightarrow GL_{11}(\mathbb{F}_{p})$ given by
$$
(v,A)\mapsto \begin{pmatrix} A& v\\ 0& 1\end{pmatrix}.
$$
Composing the two maps, we find a monomorphism $(G,\cdot)\hookrightarrow GL_{11}(\mathbb{F}_{p})$, in contradiction with 
the property of our group described in Theorem \ref{countergroup}. Thus it is impossible to define a left brace structure 
over $G$.

So it is only left to prove Theorem \ref{countergroup}. For this, we are going to apply 
Lazard's correspondence, and work in the Lie algebra setting, as we now explain.

In this section, we reduce our problem about $p$-groups to an analogous problem in the theory of 
nilpotent Lie rings, by means of the Lazard's correspondence. A good reference for this result is 
\cite[Chapters 9 and 10]{Khukhro}. Here is a summary of part of this result that we need in this paper.

\paragraph{Lazard's correspondence:} Let $p$ be a fixed prime number.
Let $\mathfrak{g}$ be a nilpotent Lie ring with additive group isomorphic to a finite $p$-group and with 
nilpotency class strictly smaller than $p$. Then the Baker-Campbell-Hausdorff formula (see \cite[Chapter 9]{Khukhro}) defines a product 
on $\mathfrak{g}$ which turns the set of elements of $\mathfrak{g}$ into a $p$-group. We denote this 
group by $\exp(\mathfrak{g})$.

Conversely, if $G$ is a finite $p$-group of nilpotency class strictly smaller than $p$, then the inversion of the 
Baker-Campbell-Hausdorff formula defines a Lie bracket on $G$ which turns the set of elements of $G$ into a nilpotent Lie ring
with additive group isomorphic to a finite $p$-group.  We denote this Lie ring by $\log(G)$. 

Moreover, $\exp$ and $\log$ are mutually inverse functors, so they define an equivalence between the category of finite 
$p$-groups with nilpotency class $c<p$, and the category of nilpotent Lie rings with additive group isomorphic 
to a finite $p$-group and with nilpotency class $c<p$.

This correspondence preserves many properties: nilpotency class, derived length, subgroups correspond to sub-Lie rings, etc. 
In particular, we have the following corollary relating the orders of the elements of $\mathfrak{g}$ and of $\exp(\mathfrak{g})$. 

\begin{corollary}\label{LazOrd}
Let $\mathfrak{g}$ and $G=\exp(\mathfrak{g})$ be as above. Then,
 $o_{(\mathfrak{g},+)}(g)=o_{(G,\cdot)}(g)$ for every $g\in G$ (observe that $\mathfrak{g}$ and $G=\exp(\mathfrak{g})$ share the 
same set of elements). 
\end{corollary}

\paragraph{} Recall that $U_{m}(\mathbb{F}_{p})$ denotes the group of upper triangular matrices of size $m$ 
with coefficients over $\mathbb{F}_p$ and with $1$'s in the diagonal . It is easy to see that $U_{m}(\mathbb{F}_{p})$ is a 
Sylow $p$-subgroup of $GL_m(\mathbb{F}_p)$, that its nilpotency class is $m-1$, and that 
$\log(U_{m}(\mathbb{F}_{p}))\cong \mathfrak{u}_{m}(\mathbb{F}_{p})$, the Lie algebra of upper triangular matrices of size 
$m$ with $0$'s in the diagonal 
with coefficients over $\mathbb{F}_p$.
Now consider the following result.

\begin{theorem}\label{counterlie}
For some primer $p\geq 12$, there exists a $\mathbb{F}_{p}$-Lie algebra $L_{p}$ of order $p^{10}$ with no Lie algebra monomorphism
$\rho: L_{p}\hookrightarrow \mathfrak{u}_{11}(\mathbb{F}_{p})$.
\end{theorem}

This theorem is enough to prove Theorem \ref{countergroup}, by the following argument.

\begin{proof}{ \bf(of Theorem \ref{countergroup})}
The Lie algebra of Theorem \ref{counterlie} has order $p^n=p^{10}$. 
Since $p> 10=n$, and the nilpotency class is always less than $n$, we can apply Lazard's correspondence to 
$L_{p}$. 
Then, this gives us a group $G=\exp(L_{p})$ of order $p^{10}$ with exponent $p$, because $(L_{p},+)\cong (\mathbb{F}_{p})^{10}$ and
$o_{(G,\cdot)}(g)=o_{(L_p,+)}(g)=p$ for all $g\in G$ by Corollary \ref{LazOrd}. 
Moreover, there is no monomorphism $(G,\cdot)\hookrightarrow GL_{11}(\mathbb{F}_{p})$ because, if 
$\Delta: (G,\cdot)\to GL_{11}(\mathbb{F}_{p})$ is an injective morphism, 
after a suitable conjugation, we may assume that $\Delta: (G,\cdot)\hookrightarrow U_{11}(\mathbb{F}_{p})$.
Then 
$\log(\Delta): L_{p}\to \log(U_{11}(\mathbb{F}_{p}))\cong \mathfrak{u}_{11}(\mathbb{F}_{p})$ (we can apply the Lazard's correspondence 
to $U_{11}(\mathbb{F}_p)$ because $p$ is bigger than its nilpotency class, which is $m-1=10$) is also an injective Lie morphism since
$\log(\Delta(x))=0$ implies $\Delta(x)=\exp(\log(\Delta(x)))=\exp(0)=\Id$, and $x=1_G=0_{L_{p}}$ because $\Delta$ is injective. 
\end{proof}

So it only remains to prove Theorem \ref{counterlie}. We will find an example of this type in the next section, 
based on analogous examples in characteristic $0$. This kind of examples was obtained in \cite{Burde} 
to disprove a conjecture of Milnor about left-invariant affine structures over manifolds \cite{Milnor}.
 After showing that this conjecture can be disproved by means of faithful Lie algebra representations, 
 the key idea in \cite{Burde} is to prove that a system of polynomial equations has no solution over $\mathbb{C}$. 
 We prove in the next section that the same system has no solution over $\mathbb{F}_p$ for some $p$.

\section{Definition of the Lie algebra counterexample}

To prove Theorem \ref{counterlie}, we find explicitly a concrete example of a Lie algebra with the desired properties, 
based on analogous examples appearing in \cite{Burde}, where Burde finds 
an infinite family of nilpotent Lie algebras over $\mathbb{C}$ of dimension $10$ with no faithful representation of dimension $11$.
This family is described as follows: if we denote the basis of the Lie algebra as $e_0,e_1,\dots,e_9$, its Lie bracket is defined by
\begin{eqnarray*}
\lie{e_0,e_i} &=& e_{i+1},\text{ for all } i=1,2,\dots,8,\\
\lie{e_1,e_2} &=& \lambda_1 e_4+\lambda_2 e_5+\cdots +\lambda_6 e_9,\\
\lie{e_1,e_3} &=& \lambda_1 e_5+\lambda_2 e_6+\cdots +\lambda_5 e_9,\\
\lie{e_1,e_4} &=& (\lambda_1-\lambda_7)e_6+(\lambda_2-\lambda_8)e_7+(\lambda_3-\lambda_9)e_8+(\lambda_4-\lambda_{10})e_9,\\
\lie{e_1,e_5} &=& (\lambda_1-2\lambda_7)e_7+(\lambda_2-2\lambda_8)e_8+(\lambda_3-2\lambda_9)e_9,\\
\lie{e_1,e_6} &=& (\lambda_1-3\lambda_7+\lambda_{11})e_8+(\lambda_2-3\lambda_8+\lambda_{12})e_9,\\
\lie{e_1,e_7} &=& (\lambda_1-4\lambda_7+3\lambda_{11})e_9,\\
\lie{e_1,e_8} &=& -\lambda_{13} e_9,\\
\lie{e_2,e_3} &=& \lambda_7 e_6+\lambda_8 e_7+\cdots +\lambda_{10} e_9,\\
\lie{e_2,e_4} &=& \lambda_7 e_7+\lambda_8 e_8+\lambda_9 e_9,\\
\lie{e_2,e_5} &=& (\lambda_7-\lambda_{11})e_8+(\lambda_8-\lambda_{12})e_9,\\
\lie{e_2,e_6} &=& (\lambda_7-2\lambda_{11})e_9,\\
\lie{e_2,e_7} &=& \lambda_{13} e_9,\\
\lie{e_3,e_4} &=& \lambda_{11} e_8+\lambda_{12} e_9,\\
\lie{e_3,e_5} &=& \lambda_{11} e_9,\\
\lie{e_3,e_6} &=& -\lambda_{13} e_9,\\
\lie{e_4,e_5} &=& \lambda_{13} e_9,
\end{eqnarray*}
and the others brackets $[e_i,e_j]$, $i>j$, are equal to zero (see \cite[page 607]{Burde}). 
There are also some relations that the $\lambda's$ must satisfy: 
$\lambda_1\neq 0$, $\lambda_7=-\lambda_1$, 
$\lambda_{11}=3\lambda_1$, 
$\lambda_{12}=-(9\lambda_2+16\lambda_8)+\frac{\lambda_{13}}{\lambda_1}(2\lambda_3+\lambda_9)$, and
$3\lambda_2+\lambda_8\neq 0$. The explanation for these relations can be found in \cite[page 607 and Proposition 6]{Burde} 
(there Burde denotes the Lie algebra by Case (A1)). Note that any of these Lie algebras has nilpotency class $9$, and 
that the center is equal to $\left\langle e_9\right\rangle$. 
In \cite[Proposition 6]{Burde}, it is proved that, for any Lie algebra $\mathfrak{g}$ in this family, there is no injective morphism
$\mathfrak{g}\hookrightarrow M_{11}(\mathbb{C})$.

To fix ideas, we choose one concrete example from this family. 
If we choose the values $\lambda_1=\lambda_2=1$, $\lambda_3=\dots=\lambda_6=0$, $\lambda_7=-\lambda_1=-1$, 
$\lambda_8=-2$, $\lambda_9=-25$, $\lambda_{10}=0$, $\lambda_{11}=3\lambda_1=3$, 
$\lambda_{13}=1$, and
$\lambda_{12}=-(9\lambda_2+16\lambda_8)+\frac{\lambda_{13}}{\lambda_1}(2\lambda_3+\lambda_9)=-2$, 
our Lie algebra $L$ is given by
\begin{eqnarray}\label{lierelations}
\lie{e_0,e_i} &=& e_{i+1},\text{ for all } i=1,2,\dots,8,\nonumber\\
\lie{e_1,e_2} &=& e_4+e_5,\nonumber\\
\lie{e_1,e_3} &=& e_5+e_6,\nonumber\\
\lie{e_1,e_4} &=& 2e_6+3e_7+25e_8,\nonumber\\
\lie{e_1,e_5} &=& 3e_7+5e_8+50e_9,\nonumber\\
\lie{e_1,e_6} &=& 7e_8+5e_9,\nonumber\\
\lie{e_1,e_7} &=& 14e_9,\nonumber\\
\lie{e_1,e_8} &=& -e_9,\\
\lie{e_2,e_3} &=& -e_6-2e_7-25e_8,\nonumber\\
\lie{e_2,e_4} &=& -e_7-2e_8-25e_9,\nonumber\\
\lie{e_2,e_5} &=& -4e_8,\nonumber\\
\lie{e_2,e_6} &=& -7e_9,\nonumber\\
\lie{e_2,e_7} &=& e_9,\nonumber\\
\lie{e_3,e_4} &=& 3e_8-2e_9,\nonumber\\
\lie{e_3,e_5} &=& 3e_9,\nonumber\\
\lie{e_3,e_6} &=& -e_9,\nonumber\\
\lie{e_4,e_5} &=& e_9,\nonumber
\end{eqnarray}
and the other brackets $[e_i,e_j]$, $i>j$, equal to zero. We have chosen specifically this example because
it coincides with the suggested Lie algebra 
to be a counterexample that can be found in \cite[Section 12]{Rump}. 

Observe that, with this values of $\lambda_1,\dots,\lambda_{13}$, we can think of $L$ as a $\Z$-Lie algebra. We will 
denote by $L_p$ its reductions modulo $p$: $L_{p}=L\otimes_\Z \mathbb{F}_p=L/pL$, which is a $\mathbb{F}_p$-Lie algebra.
One of this $L_p$ will be our example satisfying the conditions of Theorem \ref{counterlie}, as we show now. 

First, we are going to translate the problem of the determination of injective Lie morphisms from $L\otimes K$ to $\mathfrak{u}_{11}(K)$
to a problem of solving a system of polynomial equations in several variables. 
Note that any Lie morphism $\rho: L\otimes K\to\mathfrak{u}_{11}(K)$ is determined by $\rho(e_0)$ and $\rho(e_1)$, because
$\rho(e_{i+1})=\rho([e_0,e_i])=[\rho(e_0),\rho(e_i)]$. Now take two matrices $E_0, E_1\in \mathfrak{u}_{11}(K)$. 
Take the coefficients of 
$E_0$ and $E_1$ as unknown variables $y_1,\dots,y_k$, and define by induction $E_{i+1}:=[E_0,E_i]=E_0E_i-E_iE_0$. 
Then, the map $e_i\mapsto E_i$ extends to a Lie morphism if and only if 
they satisfy the relations in (\ref{lierelations}); i.e. if they satisfy 

\begin{eqnarray*}
\mlie{E_1}{E_2} &-& (E_4+E_5)=0,\\
\mlie{E_1}{E_3} &-& (E_5+E_6)=0,\\
\mlie{E_1}{E_4} &-& (2E_6+3E_7+25E_8)=0,\\
\mlie{E_1}{E_5} &-& (3E_7+5E_8+50E_9)=0,\\
\mlie{E_1}{E_6} &-& (7E_8+5E_9)=0,\\
\mlie{E_1}{E_7} &-& 14E_9,\\
\mlie{E_1}{E_8} &+& E_9=0,\\
\mlie{E_2}{E_3} &+& E_6+2E_7+25E_8=0,\\
\mlie{E_2}{E_4} &+& E_7+2E_8+25E_9=0,\\
\mlie{E_2}{E_5} &+& 4E_8=0,\\
\mlie{E_2}{E_6} &+& 7E_9=0,\\
\mlie{E_2}{E_7} &-& E_9=0,\\
\mlie{E_3}{E_4} &-& (3E_8-2E_9)=0,\\
\mlie{E_3}{E_5} &-& 3E_9=0,\\
\mlie{E_3}{E_6} &+& E_9=0,\\
\mlie{E_4}{E_5} &-& E_9=0,
\end{eqnarray*}
and $\mlie{E_i}{E_j}=0$ for the other $i,j$ with $i>j$. 
The coefficients of this matrix relations are polynomials $f_1,\dots,f_l$ in the variables $y_1,\dots,y_k$ with coefficients over $\Z$
(all the coefficients in the relations are integers, and the only operations that appear are sums and products of matrices). 
Then, 
$E_0$ and $E_1$ defines a Lie morphism $\rho: L\otimes K\to\mathfrak{u}_{11}(K)$ if and only if $f_1=\dots=f_l=0$, 
with $f_1,\dots,f_l\in \Z[y_1,\dots,y_k]$, has a solution over $K$. 

Besides, $\rho: L\otimes K\to\mathfrak{u}_{11}(K)$ is injective if and only if $\rho(e_9)\neq 0$, because 
$Z(L)=\left\langle e_9\right\rangle$, and, if the kernel of $\rho$ is non-trivial, it intersects the center 
non-trivially (any non-zero ideal in a nilpotent Lie algebra intersects the center non-trivially \cite[Corollary 1.1.15]{CG}).
 This condition can also be 
translated in terms of a system of polynomial equations: all the coefficients in $\rho(e_9)$ are polynomials in the variables 
$y_1,\dots,y_k$ with coefficients in $\Z$; 
denote them by $f'_1,\dots,f'_r$. Then, $\rho(e_9)\neq 0$ if and only if $f'_1z_1+\cdots+f'_rz_r=1$ has a solution over $K$, where 
$z_1,\dots, z_r$ are new unknown variables. 

So there exists a Lie monomorphism $\rho: L\otimes K\to\mathfrak{u}_{11}(K)$ if and only if the system of equations 
$f_1=\dots=f_l=f'_1z_1+\cdots+f'_rz_r-1=0$, with $f_1,\dots,f_l,f'_1z_1+\cdots+f'_rz_r-1\in\Z[y_1,\dots,y_k,z_1,\dots,z_r]$, 
has a solution over $K$.
The results of \cite{Burde} show that this system has no solution over 
$K=\mathbb{C}$ using Gr\"{o}bner basis. We will show that this implies the characteristic $p$ case, by the following argument.

\begin{proof}{ \bf(of Theorem \ref{counterlie})}

We have reasoned that $L_p$ has an injective morphism to $\mathfrak{u}_{11}(\mathbb{F}_p)$ if and only if 
a concrete system of polynomial equations $\overline{f_1}=\dots=\overline{f_m}=\overline{0}$, with $f_1,\dots,f_m\in\Z[x_1,\dots,x_n]$, 
has a solution over $\mathbb{F}_p^n$. But $f_1=\dots=f_m=0$ has no solution over $\mathbb{C}^n$, so by 
Hilbert's Nullstellensatz, there exists $g_1,\dots,g_m\in \mathbb{C}[x_1,\dots,x_n]$ such that $f_1g_1+\cdots+f_mg_m=1$. 
The assumption $f_1,\dots,f_m\in\Z[x_1,\dots,x_n]$ implies that we can choose $g_1,\dots,g_m\in\mathbb{Q}[x_1,\dots,x_n]$: 
if we think of $f_1g_1+\cdots+f_mg_m=1$ as an equation for the coefficients of the $g_i$'s, it gives a 
linear system of equations with coefficients in $\Z$, so it has a solution over $\mathbb{Q}$. 

Moreover, there exists a $k\in \Z$ such that $g'_i:=kg_i$ is in $\Z[x_1,\dots,x_n]$ for any $i\in\{1,\dots,m\}$. 
Hence we get $f_1g'_1+\cdots+f_mg'_m=k$ in $\Z[x_1,\dots,x_n]$.
Now take a prime number $p$ such that $\gcd(k,p)=1$; there are an infinite number of them, 
so we can take $p\geq 12$. 
Then, the reduction of $f_1g'_1+\cdots+f_mg'_m=k$ modulo $p$ is 
$\overline{f_1}\overline{g'_1}+\cdots+\overline{f_m}\overline{g'_m}=\overline{k}\neq 0$, implying 
that the system $\overline{f_1}=\dots=\overline{f_m}=\overline{0}$ has 
no solution over $\mathbb{F}_p^n$ for this prime $p$. Or, equivalently, 
there is no injective Lie morphism $L_{p}\to \mathfrak{u}_{11}(\mathbb{F}_p)$ for this prime $p$. 
\end{proof}

\begin{remark}\rm{
We shall write more explicitly the relation between \cite{Burde} and our current paper. A LSA (left symmetric algebra)
 (also known as Pre-Lie algebra) over a field $K$ is a $K$-vector space $S$ equipped with a bilinear product 
$\cdot:S\times S\to S$ 
such that $x\cdot(y\cdot z)-(x\cdot y)\cdot z=y\cdot(x\cdot z)-(y\cdot x)\cdot z$, for every $x,y,z\in S$. 
If we define $[x,y]:=x\cdot y-y\cdot x$, then $S$ is equipped with a Lie algebra structure over $K$.
 One can show (see \cite[Proposition 9.1]{Rump}) that having a LSA is equivalent to having a Lie algebra $L$ with a bijective $1$-cocycle 
 $\pi: L\to L_{\lambda}$, where $L_{\lambda}$ is the underlying vector space of $L$, with respect to 
 the left adjoint action of $L$ over $L_\lambda$
 (recall that, given a Lie algebra $L$, a vector space $V$, and a Lie morphism $\gamma: L\to End_K(V)$, a $1$-cocycle
  is a map $\pi: L\to V$ such that $\pi([x,y])=\gamma_x( \pi(y))-\gamma_y( \pi(x))$).

It was asked by Milnor \cite{Milnor} whether every solvable Lie group admits a complete left-invariant 
affine structure (Auslander in \cite[Theorem 1.1]{Auslander} proved that any Lie group admitting a complete left-invariant affine structure 
is solvable). Equivalently, he asked whether every solvable Lie algebra admits a complete 
LSA structure (see \cite[Introduction and Lemma 1]{Burde} for the equivalence of the two questions). 
For some time it was conjectured that it was always possible to define a LSA structure 
over any nilpotent Lie algebra over $\mathbb{C}$. 
But in \cite{Benoist} and \cite{Burde} it appeared the first examples answering in the negative this question.
The argument there basically makes use of the fact that an LSA of dimension $n$ has always a 
faithful Lie algebra representation of dimension $n+1$. So they find examples of Lie algebras of 
dimension $n$ with no injective morphism to $\mathfrak{gl}_{n+1}(\mathbb{C})$, with a combination of 
theoretical reductions and the resolution of a system of polynomial equations with the computer. 
Our $L\otimes \mathbb{C}$ is one of these examples. 

So, in fact, what we show
in Theorem \ref{counterlie} is that, for $p$ big enough, $L\otimes\mathbb{F}_p$ is an example of nilpotent Lie algebra 
over $\mathbb{F}_p$ without any LSA structure. Using the Lazard correspondence, and the results about the additive group 
of a left brace explained in Section $2$, we can use this result to find 
an example of a nilpotent group $G$ with no bijective $1$-cocycles to an abelian group. Recall that bijective $1$-cocycles 
of groups $\pi: G\to A$, where $A$ is an abelian group, are equivalent to left braces with multiplicative 
group isomorphic to $G$ and additive group 
isomorphic to $A$, by \cite[Theorem 2.1]{CJR}. 

Thus, summarizing, brace theory is the study of abelian bijective $1$-cocycles over groups, and LSA structures is the study of 
bijective $1$-cocycles over Lie rings and algebras. That is what makes the two theories analogous.

We hope that arguments similar to the ones used in our paper would be useful in the future to relate results 
in characteristic $0$, and results in characteristic $p$ for $p$ big enough. }
\end{remark}

\begin{remark}\label{errorRump}\rm{
Our first attempt was to use $L_p$ for $p=11$, as is suggested in \cite{Rump}, and repeat the computer calculations of \cite{Burde}. 
But, when we take $L_{11}$, and use on it (an adaptation of) the computer programs described in \cite{Burde}, 
we find a Lie monomorphism $\rho: L_{11}\hookrightarrow \mathfrak{u}_{11}(\mathbb{F}_{11})$, defined by

\setcounter{MaxMatrixCols}{20}
$$
E_0=\rho(e_0)=
\begin{pmatrix}
0& 1& 0& 0& 0& 0& 0& 0& 0& 0& 0\\
0& 0& 1& 0& 0& 0& 0& 0& 0& 0& 0\\
0& 0& 0& 1& 0& 0& 0& 0& 0& 0& 0\\
0& 0& 0& 0& 1& 0& 0& 0& 0& 0& 0\\
0& 0& 0& 0& 0& 1& 0& 0& 0& 0& 0\\
0& 0& 0& 0& 0& 0& 1& 0& 0& 0& 0\\
0& 0& 0& 0& 0& 0& 0& 1& 0& 0& 0\\
0& 0& 0& 0& 0& 0& 0& 0& 1& 0& 0\\
0& 0& 0& 0& 0& 0& 0& 0& 0& 0& 0\\
0& 0& 0& 0& 0& 0& 0& 0& 0& 0& 1\\
0& 0& 0& 0& 0& 0& 0& 0& 0& 0& 0
\end{pmatrix},
$$

$$
E_1=\rho(e_1)=
\begin{pmatrix}
0& 5& 8& 0& 0& 0& 0& 0& 0& 0& 0\\
0& 0& 0& 9& 10& 0& 0& 0& 0& 0& 0\\
0& 0& 0& 0& 4& 0& 7& 6& 5& 0& 0\\
0& 0& 0& 0& 0& 1& 3& 0& 0& 0& 0\\
0& 0& 0& 0& 0& 0& 10& 0& 8& 0& 1\\
0& 0& 0& 0& 0& 0& 0& 7& 10& 0& 8\\
0& 0& 0& 0& 0& 0& 0& 0& 2& 0& 4\\
0& 0& 0& 0& 0& 0& 0& 0& 0& 0& 0\\
0& 0& 0& 0& 0& 0& 0& 0& 0& 0& 1\\
0& 0& 0& 0& 0& 0& 0& 0& 0& 0& 0\\
0& 0& 0& 0& 0& 0& 0& 0& 0& 0& 0
\end{pmatrix}.
$$}
\end{remark}
\begin{remark}\rm{
In fact, besides the theoretical argument of this section, 
we have repeated the computer calculations of \cite{Burde} for $p=23$, and we have obtained that $L_{23}$ is a 
Lie algebra satisfying the properties of Theorem \ref{counterlie}. So $L_{23}$ is a concrete counterexample, 
independent of an unknown prime $p$.}
\end{remark}

Now it remains to find a way to find a counterexample in a more theoretical way, 
using an argument that does not depend on the calculations of a computer. Besides, 
it remains as an open problem to characterize which finite solvable groups (in particular, 
which finite $p$-groups) are isomorphic to the multiplicative group of a left brace.

\section*{Acknowledgments}
I thank S.C. Featherstonhaugh for providing me a copy of his PhD thesis, thank D. Burde 
for sending me part of the computer programs used in \cite{Burde}, and thank Alexandre Turull 
for useful comments. I would also like to thank Ferran Ced\'o for carefully reading a first draft 
of the paper and for providing many useful suggestions. 

 Research partially supported by DGI MINECO MTM2011-28992-C02-01, and by
DGI MINECO MTM2014-53644-P.

\vspace{30pt}
 \noindent \begin{tabular}{llllllll}
 D. Bachiller \\
 Departament de Matem\`atiques \\
 Universitat Aut\`onoma de Barcelona  \\
08193 Bellaterra (Barcelona), Spain  \\
dbachiller@mat.uab.cat 
\end{tabular}

\end{document}